\documentclass{amsart}[11pt,amssymb]
\DeclareMathSymbol{\twoheadrightarrow} {\mathrel}{AMSa}{"10}

\def\Q{{\mathbf Q}}

                                         \def\NS{\mathrm{NS}}

\def\Z{{\mathbf Z}}
\def\C{{\mathbf C}}

\def\F{{\mathbf F}}

\def\Gal{\mathrm{Gal}}

\def\ord{\mathrm{ord}}

\def\Fr{\mathrm{Fr}}

\def\End{\mathrm{End}}

\def\Aut{\mathrm{Aut}}

\def\I{\mathrm{Id}}

        \def\K_a{\bar{K}}

\def\dim{\mathrm{dim}}
                  \def\Spec{\mathrm{Spec}}
\def\OC{{\mathcal O}}

\def\B{{\mathfrak B}}
\def\Pf{{\mathfrak P}}

\def\X{{\mathcal X}}

\def\ZZ{{\mathfrak Z}}

 \def\X{{\mathcal X}}
                          \def\VW{{\mathrm v}}

\newtheorem{thm}{Theorem}[section]
\newtheorem{lem}[thm]{Lemma}

\theoremstyle{definition}

\newtheorem{rem}[thm]{Remark}
\newtheorem{rems}[thm]{Remarks}
                 \newtheorem{sect}[thm]{}

\title[Ordinary reduction of K3 surfaces]{Ordinary reduction of K3 surfaces}
\author{Fedor A. Bogomolov}
\address{Courant Institute of Mathematical Sciences, New York University,
New York, NY 10012, USA}
 \email{bogomolo@cims.nyu.edu}

\author{Yuri G. Zarhin}
\address{Department of Mathematics, Pennsylvania
State University, University Park, PA 16802, USA}
\email{zarhin\char`\@math.psu.edu}

\begin{document}
\maketitle

Let $K$ be a number field and $A$ an abelian variety of positive dimension over
$K$. It is well known that $A$ has good reduction at all but finitely many
(non-archimedean) places of $K$. It is natural to ask whether among those
reductions there is ordinary one. In the most optimistic form the precise
question sounds as follows.

Is it true that there exists a finite algebraic field extension $L/K$ and a
density $1$ set $S$ of places of $L$ such that $A\times L$ has ordinary good
reduction at every place from $S$?

The positive answer is known for elliptic curves (Serre \cite{Serre1972}),
abelian surfaces (Ogus \cite{Ogus}) and  certain abelian fourfolds  and
threefolds \cite{NootO,Noot,Tankeev}.

One may ask a similar question for other classes of (smooth projective)
algebraic varieties. The aim of this note is to settle this question for K3
surfaces. Recall that an (absolutely) irreducible smooth projective surface $X$
over an algebraically closed field is called a K3 surface if the canonical
sheaf $\Omega_X^2$ is isomorphic to the structure sheaf $\OC_X$ and
$H^1(X,\OC_X)=\{0\}$.

Our main result is the following statement.

\begin{thm}
\label{main} Let $X$ be a $K3$ surface that is defined over a number field $K$.
Then there exists a finite algebraic field extension $L/K$ and a density $1$
set $\Sigma(L,X)$ of (non-archimedean) places of $L$ such that $X\times_K L$
has ordinary good reduction at every place  $\VW\in \Sigma(L,X)$.
\end{thm}

\begin{rem}
 The case of Kummer surfaces follows from the
result of Ogus  concerning the existence of ordinary reductions of abelian
surfaces. When the {\sl endomorphism field} $E$ of $X\times_K\C$ \cite[Th.
1.6]{Zarhin1983} is totally real (e.g., the Picard number is odd), the
assertion of Theorem \ref{main} was proven by Tankeev \cite{Tankeev}.
\end{rem}

{\sl Acknowledgements}. The first named author (F.B.) would like to thank the
Clay Institute for financial support and Centre Di Giorgi in Pisa for its
hospitality during the work on this paper. The second named author (Y.Z.) would
like to thank Courant Institute of Mathematical Sciences for its hospitality
during his several short visits in the years 2006--2009.

\vskip 1cm

After this paper had appeared on arXiv and was submitted, we received
 a letter from professor K. Joshi who brought to our attention his joint
preprint with C.Rajan,
http://arxiv.org/abs/math/0110070. Its Section 6 contained the proof of the
existence of a positive density set of places with ordinary reduction for any
$K3$-surface  over a number field. Some of their arguments (but not all) are
similar to the ones we used. Unfortunately, this result and its generalizations
did not appear in the printed version \cite{IMRN}.

\section{K3 surfaces over finite fields}
\label{finite}

 Let $k$ be a finite field of characteristic $p$, let  $\bar{k}$
be its algebraic closure, let  $Y$ be a K3 surface defined over $k$ and
$\bar{Y}=Y\times \bar{k}$.

\begin{sect}
\label{weil} Let $\ell$ be a prime different from $p$ and
$$P_2(Y,t)=1+\sum_{i=1}^{22} a_i t^i$$
the characteristic polynomial of the Frobenius endomorphism $\Fr$ of $\bar{Y}$
in the second $\ell$-adic cohomology group $H^2(\bar{Y},\Z_{\ell})$.
 It is known (P. Deligne \cite{Deligne};
Piatetskii--Shapiro and  I.R. Shafarevich \cite{PS}) that $P_2(Y,t)$ lies in
$1+t\Z[t]$ and does not depend on the choice of $\ell$; in addition, all
reciprocal roots of $P_2(Y,t)$ have (archimedean) absolute value $q=\#(k)$. It
is also known \cite[Cor. 1.10 on p. 63]{DI} that $\Fr$ acts on the
$\Q_{\ell}$-vector space
$$H^2(\bar{Y},\Q_{\ell})=H^2(\bar{Y},\Z_{\ell})\otimes_{\Z_{\ell}} \Q_{\ell}$$
as a semisimple (i.e., diagonalizable over an algebraic closure of $\Q_{\ell}$)
linear operator.

Let us split $P_2(Y,t)$ into a product of linear factors (over $\bar{\Q}$)
$$P_2(Y,t)=\prod_{i=1}^{22} (1-\alpha_i t)$$
where $\alpha_i$ are the reciprocal roots of $P_2(Y,t)$. Clearly,
$$a_1=-\sum_{i=1}^{22} \alpha_i.$$
Let $L=\Q(\alpha_1, \dots , \alpha_{22})$ be the splitting field of $P_2(Y,t)$:
it is a finite Galois extension of $Q$. We write $\OC_L$ for the ring of
integers in $L$. Clearly, all $\alpha_i \in \OC_L$. For each field embedding $L
\hookrightarrow \C$ the image of every $\alpha_i$ has absolute (archimedean)
value $q$. This implies that if $\alpha$ is one of the reciprocal roots then
its {\sl complex-conjugate} $\bar{\alpha}=q^2/\alpha$ is also one of the
reciprocal roots; in particular, $q^2/\alpha$ also lies in $\OC_L$. It follows
that if $\B$ is a maximal ideal in $\OC_L$ that does {\sl not} lie above $p$
then $\alpha$ is a $\B$-adic unit.
\end{sect}

\begin{sect}
\label{height} In order to describe the $p$-adic behavior of the reciprocal
roots, one has to use their crystalline interpretation and use a variant of
Katz's conjecture proven in \cite[Sect. 8]{BO}.  Recall \cite[Prop. 1.1 on p.
59]{DI} that
$$h^{0,2}(\bar{Y})=1, h^{1,1}(\bar{Y})=20$$
 and the crystalline cohomology groups of $\bar{Y}$ have no torsion. Combining
Theorem 8.39 on p. 8-47 of \cite{BO} and Example 2 on p. 659 of
\cite{MazurBAMS}, one concludes (\cite[Examples on pp. 90--91]{AM}, \cite{A})
that there exists a certain invariant $h=h(Y)$ of a K3 surface $Y$ called its
{\sl height} that enjoys the following properties.

The height $h$ is either a positive integer $\le 10$ or $\infty$. A K3 surface
is called ordinary if $h=1$ and supersingular if $h=\infty$. Let $\Pf$ be (any)
maximal ideal in $\OC_L$ that lies above $p$ and let
$$\ord_{\Pf}: L^{*} \to \Q$$
be the discrete valuation map attached to $\Pf$ and normalized by condition
$$\ord_{\Pf}(q)=1.$$
If $h=\infty$ then every $\ord_{\Pf}(\alpha)=1$. If $h \le 10$ then the
sequence
$$\ord_{\Pf}(\alpha_1), \dots , \ord_{\Pf}(\alpha_{22})$$
consists of  rational numbers $(h-1)/h, 1, (h+1)/h$: both numbers $(h-1)/h$ and
$(h+1)/h$ occur $h$ times in the sequence  while the number $1$ occurs
$(22-2h)$ times.
\end{sect}

\begin{rem}
\label{ord}
\begin{enumerate}
\item [(i)] Suppose that $h=\infty$. If $\alpha$ is one of the reciprocal roots
then $\alpha/q$ is a $\Pf$-adic unit for every $\Pf$ dividing $p$. It follows
that $\alpha/q$ is unit in $\OC_L$. On the other hand, since all archimedean
absolute values of $\alpha$ are equal to $q$, we conclude that all archimedean
absolute values of $\alpha/q$ are equal to $1$. By Kronecker's theorem,
$\alpha/q$ is a root of unity.

\item [(ii)]If $h\ne 1$ then all $\ord_{\Pf}(\alpha_i)$  are positive numbers
and therefore all the reciprocal roots lie in $\Pf$.

\item [(iii)]If $h=1$ then $(h-1)/h=0$ and therefore there is exactly one
reciprocal root $\alpha$ that does not lie in $\Pf$ and this root is simple.
\end{enumerate}
\end{rem}

\begin{lem}
\label{ordinary} $Y$ is ordinary if and only if $a_1$ is not divisible by $p$.
\end{lem}

\begin{proof}
Since $a_1=-\sum_{i=1}^{22} \alpha_i$, It follows from Remark \ref{ord} that
$a_1$ does not lie in $\Pf$ if and only if $h=1$. Now one has only to recall
that $a_1 \in \Z$ and $\Z \bigcap \Pf =p\Z$.
\end{proof}

Additional information about K3 surfaces over finite fields could be found in
\cite{Nygaard, NOgus,ZarhinDuke,Yui}.

\section{A technical result}
\label{model} Let $K$ be a number field and $\OC_K$ its ring of integers. Let
$\bar{K}$ be an algebraic closure of $K$. Let $\Gal(K)=\Gal(\bar{K}/K)$ be the
absolute Galois group of $K$.
 Let $X$ be a K3 surface that is defined over a number field $K$ and let
 $\bar{X}$ be the K3 surface $X \times_K \bar{K}$ over $\bar{K}$.

 \begin{rem}
There exists a nowhere vanishing regular  exterior $2$-form on $\bar{X}$ that
is defined over $K$. Indeed, pick any regular nowhere vanishing $2$-form
$\bar{\omega}$ on $\bar{X}$. Then for each $\sigma \in \Gal(K)$ the $2$-form
$\sigma \bar{\omega}$ coincides with $c_{\sigma}\cdot \bar{\omega}$ for a
certain non-zero $c_{\sigma}\in \bar{K}^{*}$. We get a Galois cocycle $\sigma
\mapsto c_{\sigma}$. By Hilbert's Theorem 90, there exists $a \in \bar{K}^{*}$
such that
$$c_{\sigma}=\sigma(a)/a \quad \forall \ \sigma \in \Gal(K).$$
It follows that the $2$-form $\omega=a^{-1}\bar{\omega}$ is Galois-invariant.
As a corollary, we obtain that the canonical (invertible) sheaf
$\Omega^2_{X/K}$ is isomorphic to the structure sheaf $\OC_{X}$.
 \end{rem}

 If $S$ is a finite set of primes then let us consider the localization
 $\Lambda=\Lambda_S:=\OC_K[S^{-1}]$  of $\OC_K$ with respect to $S$.
 Clearly, $\Lambda$ is a Dedekind ring,
$$\OC_K\subset \Lambda \subset K$$
and $\Spec(\OC_K)\setminus \Spec(\Lambda)$ is a finite set of maximal ideals in
$\OC_K$, whose residual characteristic lies in $S$.

\begin{sect}
 There exists a a finite set of primes $S$ and a smooth projective morphism $\X \to \Spec(\Lambda)$ that
 enjoy the following properties:

 \begin{itemize}
\item The generic fiber $\X_K$ coincides with $X$.

\item The invertible (canonical) sheaf $\Omega^2_{\X/\Lambda}$ is isomorphic to
the structure sheaf $\OC_{\X}$.

\item The cohomology group $H^1(\X, \OC_{\X})$ is a free $\Lambda$-module of
finite rank.

\item For every commutative $\Lambda$-algebra $B$
 the canonical map
$$H^1(\X, \OC_{\X})\otimes_{\Lambda} B \to H^1(\X_B, \OC_{\X_B})$$
is an isomorphism. Here
$$\X_B=\X\times_{Spec(\Lambda)} \Spec(B).$$
Since $H^1(\bar{X},\OC_{\bar{X}})=\{0\}$, we conclude that $H^1(\X,
\OC_{\X})=\{0\}$ and therefore $H^1(X_s,\OC_{X_s})=\{0\}$
  for any geometric point $s$ of
$\Spec(\Lambda)$.  In particular, $\X_s$ is a K3 surface.
 \end{itemize}
\end{sect}

The assertion follows from general results about the existence of smooth
projective models \cite[pp. 157--158, Prop. A.9.1.6]{HS} and base change
theorems for Hodge cohomology \cite[Sect. 8, pp. 203--205]{Katz} (see also
\cite[Sect. 5]{Mumford}).

 We call such schemes $\X \to \Spec(\Lambda)$ {\sl good modeles} of $X$.

 \section{Ordinary reductions}
 Let $X$ be a K3 surface over a number field $K$.
 Let us pick a prime $\ell>2 \cdot 22$ and consider the corresponding $22$-dimensional $\ell$-adic
 representation \cite{Serre1972,SerreKyoto}

 $$\rho_{2,X}: \Gal(K) \to \Aut_{\Z_{\ell}}(H^2(\bar{X}, \Z_{\ell}))\subset
 \Aut_{\Q_{\ell}}(H^2(\bar{X}, \Q_{\ell})).$$

 We write $G_{\ell,X,K}$ for the image $\rho_{2,X}(\Gal(K))$: it is a closed
 compact subgroup of $\Aut_{\Z_{\ell}}(H^2(\bar{X}, \Z_{\ell}))$; in particular,
 it is an $\ell$-adic Lie subgroup of $\Aut_{\Q_{\ell}}(H^2(\bar{X},
 \Q_{\ell}))$. One may view $G_{\ell,X,K}$ as the Galois group of the infinite Galois
 extension $\bar{K}^{\ker(\rho_{2,X})}/K$ where $\bar{K}^{\ker(\rho_{2,X})}$
 is the subfield of $\ker(\rho_{2,X})$-invariants in $\bar{K}$.

 We write $\I$ for the identity automorphism of $H^2(\bar{X},
 \Z_{\ell})$. Then the set
 $$\ZZ:=\{c\cdot \I\mid c \in \Z_{\ell}^*\}\bigcap G_{\ell,X,K}$$
 is a closed normal $\ell$-adic Lie subgroup in $G_{\ell,X}$.

 \begin{lem}
 \label{thin}
 The subgroup $\ZZ$ is not open in $G_{\ell,X,K}$. In particular,
 $\dim(\ZZ) < \dim(G_{\ell,X})$.
 \end{lem}

 \begin{proof}
Suppose that $\ZZ$ is  open in $G_{\ell,X}$. Then it has finite index and there
exists a finite Galois extension $K^{\prime}/K$ such that
$$G_{\ell,X,K^{\prime}}=\rho_{2,X}(\Gal(K^{\prime}))=\ZZ,$$ i.e.,
$\Gal(K^{\prime})$ acts on $H^2(\bar{X}, \Q_{\ell})$ via scalars. It follows
that $\Gal(K^{\prime})$ acts on the {\sl twisted} $\ell$-adic cohomology group
$H^2(\bar{X}, \Q_{\ell})(1)$ also via scalars. This implies that $H^2(\bar{X},
\Q_{\ell})(1)^{\Gal(K^{\prime})}$  is either $H^2(\bar{X}, \Q_{\ell})(1)$ or
zero. However, it is known \cite[Th. 5.6 on p. 80]{Tate} that
$$H^2(\bar{X},
\Q_{\ell})(1)^{\Gal(K^{\prime})}=\NS(\bar{X})^{\Gal(K^{\prime})}\otimes
\Q_{\ell}.$$ (This assertion is the Tate conjecture for K3 surfaces that
follows from the corresponding results of Faltings  concerning abelian
varieties \cite{F}.)

Since the N\'eron--Severi group $\NS(\bar{X})$ of $\bar{X}$ is a  (non-zero)
free commutative group of rank $\le 20<22$, we conclude that $$H^2(\bar{X},
\Q_{\ell})(1)^{\Gal(K^{\prime})}\ne H^2(\bar{X}, \Q_{\ell})(1)$$ and therefore
$$H^2(\bar{X}, \Q_{\ell})(1)^{\Gal(K^{\prime})}=\{0\}.$$
However, this is not the case, because there is a hyperplane section of $X$
that is defined over $K^{\prime}$ (and even over $K$) and its $\ell$-adic
cohomology class is Galois-invariant and not zero. The obtained contradiction
proves the Lemma.
 \end{proof}

 \begin{rems}
Combining results of \cite{Deligne} and \cite{Bogomolov,BogomolovI}, one may
prove that $\dim(\ZZ)=1$.

\noindent We refer to \cite{SZ} for other applications of the Tate conjecture
\cite{Tate} and its variants to arithmetic of K3 surfaces over number fields.

 \end{rems}

The following statement and its proof are inspired by results of N. Katz and A.
Ogus \cite[Prop. 2.7.2  on p. 371]{Ogus}.

\begin{lem}
\label{ogus2}

Suppose that a prime $p$ and an element $u \in G_{\ell,X,K}$ enjoy the
following properties:

\begin{itemize}
\item[(i)]

$p-1$ is divisible by $\ell$.

\item[(ii)]

$u \in \I+ \ell \cdot\End_{\Z_{\ell}}(H^2(\bar{X}, \Z_{\ell}))$.

\item[(iii)]

The characteristic polynomial
$$P_u(t)=\det(1-tu, H^2(\bar{X}, \Q_{\ell}))=1+b_1 t + \cdots + b_{22} t^{22}$$ lies
in $\Z[t]$.

\item [(iv)]Let us split $P_u(t)$ into a product of linear factors
$$P_u(t)=\prod_{i=1}^{22} (1-\beta_i t).$$
Then all the  reciprocal roots $\beta_1, \dots , \beta_{22}$ of $P_u(t)$ have
the same  archimedean absolute value $p$.

\item [(v)] $b_1$ is divisible by $p$.
\end{itemize}

Then $p^{-1} u$ is a unipotent linear operator in $H^2(\bar{X}, \Q_{\ell})$. In
particular, if $u$ is semisimple then $u=p\cdot\I$.
\end{lem}

\begin{proof}
So, $-b_1=pc$ for some integer $c$. Notice that $\beta_1, \dots , \beta_{22}$
are the eigenvalues of $u$ and $-b_1 =\sum_{i=1}^{22}\beta_i$ is the trace of
$u$.

The congruence condition for $u$ implies that all $(\beta_i-1)/\ell$ are
algebraic integers and therefore the integer $-b_1=pc$ is congruent to $22$
modulo $\ell$. Since $p-1$ is divisible by $\ell$, it follows that $c$ is
congruent to $22$ modulo $\ell$.

It is also clear that the absolute value of $(-b_1)$ does not exceed $22\cdot
p$ and therefore $\mid c \mid \le 22$. Taking into account that $\ell>2 \cdot
22$ and $c$ is congruent to $22$ modulo $\ell$, we conclude that $c=22$, i.e.,
$-b_1 =22\cdot p$. Since $-b_1$ is a sum of $22$ complex numbers $\beta_1,
\dots , \beta_{22}$ of absolute value $p$, it follows that all $\beta_i=p$,
i.e., all eigenvalues of $u$ are equal to $p$, which means that $p^{-1} u$ is
unipotent.
\end{proof}

\begin{sect}
\label{compare}
 Choose a good model $\X \to \Spec(\Lambda)$  of $X$ (as in Sect. \ref{model}). Let
 $\VW \in \Spec(\Lambda)$ be a closed point, whose residual characteristic $p=p(\VW)$ is different from
 $\ell$. Then $\rho_{2,X}$ is unramified at $\VW$ and one may associate to $\VW$
 a Frobenius element $\Fr_{\VW} \in G_{\ell,X,K}$ (\cite{Serre1972, SerreKyoto}, \cite[Sect. 4]{Zarhin})
  that is defined up to conjugacy.
 Let us consider the corresponding closed fiber $\X(\VW)$, which  is a K3
 surface over the (finite) residue field $k(\VW)$. Let $\overline{k(\VW)}$ be an algebraic
 closure of $k(\VW)$ and $\overline{\X(\VW)}=\X(\VW)\times_
 {k(\VW)}\overline{k(\VW)}$. The Frobenius endomorphism
 $\Fr:\overline{\X(\VW)} \to \overline{\X(\VW)}$ acts on $H^2(\overline{\X(\VW)},
 \Q_{\ell}))$ and there exists an isomorphism of $\Q_{\ell}$-vector spaces
 $$H^2(\overline{\X(\VW)},
 \Q_{\ell}))\cong  H^2(\bar{X}, \Q_{\ell})$$
 such that the action of $\Fr$ becomes the action of $\Fr_{\VW}^{-1}$ (\cite{Serre1972, SerreKyoto}, \cite[Sect. 4]{Zarhin});
 in particular, we have the coincidence of the corresponding characteristic
 polynomials, i.e.,
$$P_2(X(\VW),t)=\det(1-t \Fr,H^2(\overline{\X(\VW)},
 \Q_{\ell})))=\det(1-t \Fr_{\VW}^{-1}, H^2(\bar{X}, \Q_{\ell})).$$
 In particular, the reciprocal roots of $P_2(X(\VW),t)$ are exactly the
 eigenvalues of $\Fr_{\VW}^{-1}$; in addition, it follows from the semisimplicity
 of the Frobenius endomorphism (Subsect. \ref{weil}) that $\Fr_{\VW}^{-1}$ and (therefore) $\Fr_{\VW}$ are semisimple
 linear operators in $H^2(\bar{X}, \Q_{\ell})$. So, if
$$P_2(X(\VW),t)=1+\sum_{i=1}^{22} a_i(\VW) t^i\in \Z[t]$$
then the integer $-a_1(\VW)$ coincides with the trace  of $\Fr_{\VW}^{-1}$ in
$H^2(\bar{X}, \Q_{\ell})$. It follows from Lemma \ref{ordinary} that $X(\VW)$
is ordinary if and only if $a_1(\VW)$ is {\sl not} divisible by $p(\VW)$.
\end{sect}

\begin{sect}
\label{ogusserre}
 Suppose that
$$G_{\ell,X,K}\subset  \I+ \ell \End_{\Z_{\ell}}(H^2(\bar{X}, \Z_{\ell})),$$
 $K$ contains a primitive $\ell$th root of unity. Suppose also that the residue field
$k(\VW)$ is a prime (finite) field $\F_p$ of characteristic $p=p(\VW)$. Then
$p-1$ is divisible by $\ell$.

Let us assume that $a_1(\VW)$ is divisible by $p$.  Using  the results of
 Subsect. \ref{weil} and  \ref{compare}, we may apply Lemma
\ref{ogus2}  to $u=\Fr_{\VW}^{-1}$ and conclude that $\Fr_{\VW}^{-1}=p\cdot\I$,
i.e., $\Fr_{\VW}=p^{-1}\cdot \I\in \ZZ$.

This proves that if (instead) we assume that
 $\Fr_{\VW}$ does {\sl not} belong to $\ZZ$ then $a_1(\VW)$ is {\sl not} divisible by $p$
and, thanks to the last assertion of Subsect. \ref{compare},
 $X(\VW)$ is
ordinary.

 It is well-known \cite[Theorems 1.112 and 1.113 on p. 83]{Koch} that the set of $\VW$'s
with prime residue fields has density one. On the other hand, a result of Serre
\cite[Sect. 4.1, Cor. 1]{SerreChebotarev} applied to $G=G_{\ell,X,K}$ and
$C=\ZZ$ and  combined with Lemma \ref{thin} implies that the set of $\VW$'s
with $\Fr_{\VW}\in \ZZ$ has density zero. It follows that (under our
assumptions on $K$) the set of $\VW$'s with ordinary reduction $X(\VW)$ has
density one.
\end{sect}

{\bf Proof of Theorem \ref{main}}. So, $K$ is an arbitrary number field. There
exists a finite Galois extension $L/K$ such that $L$ contains a primitive
$\ell$th root of unity and
$$G_{\ell,X,L}\subset  \I+ \ell \cdot\End_{\Z_{\ell}}(H^2(\bar{X}, \Z_{\ell})).$$
Now the result follows from the last assertion of Subsect. \ref{ogusserre}
applied to $L$ (instead of $K$.)

\end{document}